\documentclass{article}
\usepackage{amsthm}
\usepackage{amsmath}
\usepackage{amsfonts}
\usepackage{amssymb}
\usepackage{empheq}
\usepackage{xcolor,dsfont}

\usepackage[margin=25mm]{geometry}

\definecolor{gray}{HTML}{d2d2d2}

\newtheorem{thm}{Theorem}
\newtheorem{df}[thm]{Definition}
\newtheorem{ex}[thm]{Example}
\newtheorem{lemma}[thm]{Lemma}
\newtheorem{prop}[thm]{Proposition}

\newtheorem{rem}[thm]{Remark}

\newcommand{\I}{\mathcal{I}}
\newcommand{\R}{\mathds{R}}
\newcommand{\de}{\partial}
\newcommand{\J}{\mathcal{J}}
\newcommand{\s}{\mathcal{S}}
\newcommand{\pk}{para-K\"ahler }

\newcommand{\ord}{\mathrm{ord}}

\newcommand{\Z}{\mathds{Z}}
\newcommand{\N}{\mathds{N}}

\newcommand{\D}{\mathds{D}}
\newcommand{\DP}{\mathds{D}\mathds{P}}
\newcommand{\dd}{\mathcal{D}}

\title{Para-K\"ahler immersions in \pk space forms} \author{Gianni Manno, Filippo Salis }

\begin{document}

\maketitle

\begin{abstract}
In this paper, we provide necessary and sufficient conditions for the existence of \pk immersions in \pk space forms. As a consequence, we prove that, in general, a local \pk immersion cannot be globally extended, even if it is defined on a simply connected \pk manifold. Finally, we classify \pk immersions between \pk space forms.
\end{abstract}



\section{Introduction}

\subsection{The context and description of the problem}
An \emph{almost para-complex manifold} is a $2n$-dimensional manifold $M$ provided with a field of endomorphisms $\mathcal{T}$ such that $\mathcal{T}^2=1$, having eigenvalues $1$ and $-1$, whose associated eigendistributions are $n$-dimensional. If, furthermore, such distributions are integrable, which is equivalent to requiring the vanishing of the Nijenhuis tensor of $\mathcal{T}$, an almost para-complex manifold is called a \emph{para-complex manifold}.  A \emph{\pk manifold} is a para-complex manifold with a pseudo-Riemannian metric $g$ against which the field of endomorphism $\mathcal{T}$ is parallel, or, equivalently, the $2$-form $\omega=g\left(\mathcal{T}(\cdot),\,\cdot\,\right)$ is symplectic: in this case, the aforementioned distributions give a couple of Lagrangian foliations. These are important aspects for which \pk geometry and, more in general, para-complex geometry is an area of research rather active,  as evidenced by lots of papers in this field. In this regard, the survey  \cite{gm}  shows a large spectrum of applications of this geometry, mainly focused on  actions of Lie groups on the aforementioned manifolds, together with a historical introduction. One can also consult  \cite{al,amt,amt2} for further study and their bibliography for further readings.
For different applications, one can see \cite{hl} and references therein, where, among other things, it was shown a direct relation between Lagrangian submanifolds of \pk manifolds  and the Monge-Kantorovich mass transport problem.

The formal analogy with the K\"ahler geometry makes it possible to formulate problems also in the \pk context by translating the original ones. For instance, a classical  problem in K\"ahler geometry is the characterization of holomorphic and isometric immersion  into complex space forms, i.e. into   K\"ahler  manifolds of constant holomorphic sectional curvature (see \cite{Cal} and for a modern introduction to this subject \cite{loizedda}).
In particular, Calabi provided in \cite{Cal} a distinguished function, called \emph{diastasis}, whose Taylor's expansion facilitated the study of the aforementioned immersions, allowing the author to achieve some remarkable results. The same problem can be formulated in the \pk case, i.e., to characterize para-holomorphic and isometric immersions, called \emph{\pk immersions},  into \emph{\pk space forms}, i.e., \pk manifolds having constant para-holomorphic sectional curvature (namely the para-complex counterpart of the holomorphic sectional curvature).

Since analiticity of homolorphic functions plays a key role in the aforementioned Calabi's work, his techniques cannot be applied in the considered context. 
Indeed, even though its formal similarity with complex geometry, para-complex geometry is quite different, since para-holomorphic functions are not, in general, analytic but only $\mathcal{C}^\infty$-smooth. 
Therefore, even if inspired by the Calabi's work,  in order to achieve  necessary and sufficient conditions for the existence of local \pk immersions into \pk space forms,  we need to develop some different ideas, which we will explain more in details in the next section.

%
%
%
%

\subsection{Description of the paper and main results}

In Section \ref{basics}, after recalling some fundamental aspects of para-complex geometry, we introduce the diastasis function for \pk manifolds inspired by a Calabi's idea. We end the section with the description of \pk manifolds of constant para-holomorphic curvature.

Section \ref{fullpk} is mainly devoted to find necessary and sufficient conditions for the existence of  \pk immersions into \pk space forms. In particular, the main achievement of the Section \ref{rig} is Theorem \ref{rigidity} where, essentially, we prove that, if it exists, a \pk immersion is locally unique up an isometry of the ambient space. The example of \pk manifold described in Section \ref{counterex} helps to better illustrate  such local character of \pk immersions. Moreover, differently from the  K\"ahler case, this example shows that, in general, local \pk immersions into \pk space forms and defined on a simply connected manifold, cannot  be globally extended. In Section \ref{pk immersion}, taking into account the results obtained so far, we  arrive to the main achievements of the paper, i.e. Theorem \ref{main}. More precisely, we found algebraic necessary and sufficient conditions for the existence of \pk immersions in \pk space forms. Such conditions are expressed in terms of a suitable constructed function and its derivatives and, being algebraic, they  can  be straightforwardly computed.

Finally, in Section \ref{pkspforms}, by means of the results  obtained in the previous sections, we get a complete classification of \pk immersions between \pk space forms.

\subsection*{Notation}
A multi-index $I=(i_1,\dots,i_n)$ is an element of $\N^n$ and its length $|I|$ is defined as the number $|I|:=\sum_{k=1}^ni_k$.  If $(x_1,\dots,x_n)$ are local coordinates, we define the
derivative operators $\frac{\de^{|I|}}{\de x^I}$ as follows:
$$
\frac{\de^{|I|}}{\de x^I}:=\frac{\de^{|I|}}{\de x_1^{i_1}\cdots \de x_n^{i_n}}\,.
$$
We fix a total order on $\N^n$ such that $I_0=0\in\N^n$ and  $|I_i|\leq|I_{i+1}|$ for any $i\in\N$. This obviously induces a total order on the set of the aforementioned derivative operators $\frac{\de^{|I|}}{\de x^I}$. Once a total order on $\N^n$ is fixed, one can construct a bijection  $\iota$ between any subset $\I\subset\N^n$ with finite cardinality and $\{1,\mathellipsis, \#\I\}\subset\N$ preserving the total orders:
\begin{equation}\label{iota}
\iota :I\in\I\longrightarrow \iota(I)\in\{1,\mathellipsis, \#\I\}.
\end{equation}

\section{Basics of para-complex geometry}\label{basics}
\subsection{Para-holomorphic  functions}

The 2-dimensional algebra over $\R$ of \emph{para-complex numbers} $\D$ is generated by $1$ and $\tau$, where 
$$\tau^2=1\,.
$$
In  analogy with the complex numbers, we are going to adopt the notation used in \cite{hl}: each $z\in\D$ can be written as 
$$
z=x+\tau y\,,
$$
and we are going to refer to $x$ and $y$ as the \emph{real} and \emph{imaginary part} of $z$, respectively.
In analogy with the complex numbers, we define the conjugate of $z$
$$ \bar z=x-\tau y$$
and 
$$|z|^2=z\bar z=x^2-y^2.$$
For later purposes, it is useful to introduce also another coordinate system on $\D$, described as follows. We switch the basis $(1,\tau)$ with $(e,\bar e)$, where
$$
e=\frac{1}{2}\left(1-\tau\right)\,,\qquad \bar e=\frac{1}{2}\left(1+\tau\right)\,,
$$
and we are going to say that $(u,v)$ are \emph{null-coordinates} of $z$ if  $z=ue+v\bar e$.

Now we can translate on $\D^n$  what we said about $\D$. In particular, for any $z,w\in\D^n$, we define 
$$\langle z, w\rangle= \sum_{i=1}^n z_i\bar w_i$$
and
$$\|z\|^2:=\sum_{i=1}^n |z_i|^2.$$

\begin{df}
A  function
$$
\begin{array}{rrcl}
F:&U\subseteq \D^n&\longrightarrow&\D\\
& (z_1,\mathellipsis,z_n)&\longmapsto& g(x_1,y_1,\mathellipsis,x_n,y_n)+ \tau\ h(x_1,y_1,\mathellipsis,x_n,y_n),
\end{array}$$ 
where $z_i=x_i+\tau y_i$, is called \emph{para-holomorphic} if and only if $g$ and $h$ are smooth and
\begin{equation}\label{dezbar}
\frac{\de F}{\de\bar z_i}:=\frac{1}{2}\left( \frac{\de g}{\de x_i}-\frac{\de h}{\de y_i}\right)+\frac{\tau}{2}\left( \frac{\de h}{\de x_i}-\frac{\de g}{\de y_i}\right)=0\end{equation}
for any $1\leq i\leq n$.
\end{df}
The number $n$ stands for the \emph{``para-complex" dimension} of $\D^n$.

In analogy to the complex setting, the differential operator $\frac{\de}{\de z_i}$ is defined by
\begin{equation}\label{dez}
\frac{\de F}{\de z_i}:=\overline{\left( \frac{\de \bar F}{\de \bar z_i}\right)}.
\end{equation}
\begin{rem}\label{ph1}
Let  $F:U\subseteq \D^n\to\D$ be a para-holomorphic function. By considering the null-coordinates $(\xi,\eta)=(\xi_1,\mathellipsis,\xi_n,\eta_1,\mathellipsis,\eta_n)$ on $\D^n$ and by writing $F$ as
$$
F(\xi_1e+\eta_1\bar e,\mathellipsis,\xi_n e+\eta_n\bar e)=u(\xi_1,\eta_1,\mathellipsis,\xi_n ,\eta_n)\ e+ v(\xi_1,\eta_1,\mathellipsis,\xi_n ,\eta_n)\ \bar e\,,
$$ 
where $u$ and $v$ are real functions on an open subset of $\R^{2n}$, we straightforwardly get
$$
\frac{\de F}{\de\bar z_i}=\frac{\de u}{\de \eta_i}e+\frac{\de v}{\de \xi_i}\bar e\,.
$$
Therefore, $F$ is para-holomorphic if and only if $u$ is independent of $(\eta_1,\mathellipsis,\eta_n)$  and $v$ is independent of  $(\xi_1,\mathellipsis,\xi_n)$.
\end{rem}
\begin{df}\label{ph2}
A function 
$$
\begin{array}{rccc}
F:&U\subseteq \D^n&\longrightarrow&\D^m\\
& z=(z_1,\mathellipsis,z_n)&\longmapsto&\big(\, f_1(z),\mathellipsis,f_m(z)\,\big)
\end{array}
$$ 
is para-holomorphic if and only if each component $f_i$ is para-holomorphic.
\end{df}

\subsection{Para-K\"ahler  manifolds, diastasis and \pk space forms}\label{sec.pk.diastasis}

The condition of integrability of almost para-complex structures, that we wrote in the Introduction, is equivalent to the existence of a para-holomorphic atlas. More precisely, we give the following definition.
\begin{df}
A smooth manifold $M^n$ of para-complex dimension $n$ is called \emph{para-complex} if it admits an atlas of para-holomorphic coordinates $(z_1,\mathellipsis,z_n)$, such that the transition functions are para-holomorphic.
\end{df}

Differently from what we did in the Introduction,  for our purposes, below we define \pk manifolds by introducing local potentials.

\begin{df}\label{pk}
A \emph{\pk manifold}  of para-complex dimension $n$ is a para-complex manifold $M$ endowed with a symplectic form $\omega$ (called \emph{\pk form}) such that, for any point $p\in M$, there exists an open neighborhood $U\ni p$ and  a smooth function $\Phi:U\to \R$ (called \emph{\pk potential}) satisfying
$$\omega|_{U}= \frac{\tau}{2}\de\bar\de\Phi:=\frac{\tau}{2}\sum_{i,j=1}^n\frac{\de^2\Phi}{\de z_i\de \bar z_j}dz_i\wedge d\bar z_j.$$ 
\end{df}
\begin{rem}
To a given \pk manifold $(M ,\omega)$  of para-complex dimension $n$ it is associated a pseudo-Riemannian  metric $g$ on $M$. Indeed, if $\omega$ admits a \pk potential $\Phi$ in an open set $U\subset M$, the restriction of $g$ to $U$ can be defined as the real part of 
$$\sum_{i,j=1}^n\frac{\de^2\Phi}{\de z_i\de \bar z_j}dz_i\otimes d\bar z_j.$$ 
\end{rem}
Let  $U$ be an open subset of a \pk manifold  $(M,\omega)$,  where it is defined a local potential $\Phi$. We  assume that $U$ can be covered by a system of null-coordinates 
$$(\xi,\eta)=(\xi_1,\mathellipsis,\xi_n,\eta_1,\mathellipsis,\eta_n).$$
 Moreover, up to shrinking $U$, we can also assume that it splits as a product 
$$U=\Omega\times \Omega.$$  
 We define the \emph{diastasis function}
$$D:U\times U=\Omega\times\Omega\times\Omega\times\Omega\to\R$$
as
\begin{equation}\label{diast}
D\left(\xi,\eta,\zeta,\lambda\right)=\Phi(\xi,\eta)-\Phi(\zeta,\eta)-\Phi(\xi,\lambda)+\Phi(\zeta,\lambda).
\end{equation}
Nevertheless, even if the previous definition is given in local coordinates, it gives rise to a well defined function on a neighborhood of the diagonal of $M\times M$. More precisely, we have the following proposition.
\begin{prop}\label{uniqueness}
The diastasis function is a function defined in a neighborhood of the diagonal of the product manifold $M\times M$. In particular, it does not depend on the choice of the \pk potential.
\end{prop}
\begin{proof}
Let $(p,p)\in M\times M$.
Let $\Phi$ and $\tilde \Phi$ be two local potentials defined on the same open subset $U\ni p$ of a \pk manifold. By definition of potential (see Definition \ref{pk}) 
 we have
$$\de\bar\de(\Phi-\tilde\Phi)=0,$$
namely, if we fix some null-coordinates $(\xi,\eta)$ on $U$, 
$$\frac{\de^2(\Phi-\tilde\Phi)}{\de\xi_i\de\eta_j}=0,\qquad \forall 1\leq i,j\leq n.$$
Hence, there exist two functions $F,\ G\in\mathcal{C}^\infty(\Omega)$ such that 
$$\tilde \Phi (\xi,\eta)=\Phi(\xi,\eta)+F(\xi)+G(\eta).$$
Our statement  follows by the definition of diastasis function, see \eqref{diast}.
\end{proof}
\begin{rem}\label{D0}
Let $(M,\omega)$ be a \pk manifold with local potential $\Phi$, defined on an open subset $U\subseteq M$. Then the function 
$$\begin{array}{rccc}
D_{p}: & U&\longrightarrow & \R\\
&q&\longmapsto&D(p,q)
\end{array}$$ is also a local potential for $(M,\omega)$.
\end{rem}

In complete analogy with the case of K\"ahler manifolds, one can define the para-holomorphic sectional curvature (see for instance \cite{gm}) and, so, the \pk space forms.
\begin{df}
A \emph{\pk space form} is a \pk\ manifold with constant para-holomorphic sectional curvature. 
We denote with $\s_c^N$ an $N$-dimensional simply connected para-complex manifold that can be endowed  with a \pk form $\omega_c$ whose associated pseudo-Riemannian metric $g_c$ is complete and has constant para-holomorphic sectional curvature equal to $c$.
\end{df}

\begin{prop}[\cite{gm} Prop. 3.11]
Two complete and simply connected \pk space forms with the same para-holomorphic sectional curvature are para-holomorphically isometric. 
\end{prop}

\noindent By \cite{gm}, we have that  an open subset of a \pk space form is para-holomorphically isometric to an open subset of one  of the subsequent models, according to their (constant) para-holomorphic sectional curvature. Therefore, since we mainly interested in local \pk immersions into open subsets of \pk\ space forms, we are going to assume that  $\left(\s_c^N,\omega_c\right)$  is  one of the following models.

\begin{description}
\item[Flat case:] The model of the flat \pk space
\begin{equation}\label{omega0}
\left(\s_0^N,\omega_0\right)=\left(\D^N,{\tau}\de\bar\de \|z\|^2\right),
\end{equation}
whose potential $\Phi$ reads in null-coordinates $(\xi_1,\mathellipsis, \xi_N,\eta_1,\mathellipsis,\eta_N)$ as 
\begin{equation}\label{dn}
 \Phi=4\sum_{i=1}^N \xi_i\eta_i =D_0(\xi,\eta), \end{equation}
where $D_0$ is defined in Remark \ref{D0},
is an example of homogeneous space with respect to its para-holomorphic isometry group. Such group consists of translations and $\D$-unitary transformations $z\in\D^N\to Az\in \D^N$, where $A$ belongs to the \emph{$\D$-unitary group} $$\mathrm{U}_N(\D)=\{A\in \D^{N,N}\ |\ \|Aw\|^2=\|w\|^2\ \forall w\in\D^N\}.$$
\item[Non-flat cases:]
Similarly to the real and complex setting, the para-complex projective space $\DP^N$ can be defined as the quotient of
$$\{Z\in\D^{N+1} \ |\ \|Z\|^2>0\}$$
under the equivalence relation given by
$Z\sim W$ if and only if there exists $\alpha\in\D$ such that $Z=\alpha W$ with $|\alpha|^2>0$.

Our model of non-flat \pk space form will be
\begin{equation}
(\s_c^N,\omega_c)=\left(\DP^N,\frac{4\tau}{c}\de\bar\de\log \|Z\|^2 \right).
\end{equation}
In null-coordinates  $(\xi_1,\mathellipsis, \xi_N,\eta_1,\mathellipsis,\eta_N)$ of the affine chart $\mathcal{U}_\alpha:=\{[Z_0,\mathellipsis,Z_N]\in \DP^N \ |\ |Z_\alpha|^2\neq 0\}$, where $\alpha=1,\mathellipsis,n$, i.e., $\xi_i e+\eta_i \bar e= \frac{Z_i}{Z_\alpha}$ for any $i\neq \alpha$, the potential $\Phi$ is equal to 
\begin{equation}\label{dpn}
\Phi= \frac{8}{c}\log\left( 1+2\sum_{i=1}^N \xi_i\eta_i\right)=D_0(\xi,\eta). \end{equation}
Since the action of $\mathrm{U}_{N+1}(\D)$ passes to the quotient, we can easily see that these \pk space forms are homogeneous with respect to the action of their para-holomorphic isometry groups.
\end{description}

\section{Para-K\"ahler immersions and full \pk immersions in space forms}\label{fullpk}
\begin{df}
Let  $(S,\theta)$  and $(M,\omega)$ be two \pk manifolds. A \emph{\pk immersion} of $(S,\theta)$ into $(M,\omega)$ is a para-holomorphic immersion $f: S\to M$ such that $f^*(\omega)=\theta$.
\end{df}
If $h$ and $g$ are, respectively, the  pseudo-Riemannian metric associated to $\theta$ and $\omega$, then we also have that $f^*(g)=h$.

\begin{prop}[Hereditary Property]\label{her}
Let $(S,\theta)$  and $(M,\omega)$ be two \pk manifolds and  let 
$$f: (S,\theta)\to(M,\omega)$$ be a \pk immersion. Let $D^S$ and $D^M$ be  the diastasis functions of $S$ and $M$, respectively. 
If $p\in S$  and if $D^S_p$ (cfr. Remark \ref{D0}) is defined on an open subset $U\ni p$ of $S$, then
$$D^S_p(q)=D^M_{f(p)}\left(f(q)\right),\qquad \forall q\in U.$$
\end{prop}
\begin{proof}
If $\Phi^S$ is a potential for $(S,\theta)$ around $p$ and $\Phi^M$ is a potential for $(M,\omega)$ around $f(p)$, then, by taking into account that $f^*(\omega)=\theta$, we have that
$$\de\bar\de\left( \Phi^M\circ f-\Phi^S\right)=0.$$
Therefore, the statement of the proposition follows from the diastasis' definition, see \eqref{diast}.
\end{proof}
Diastasis function, together with its properties, in particular the hereditary one, will be used to characterize \pk Einstein manifolds admitting a \pk immersion into a \pk space form, following the ideas present in \cite{ms1,ms2}.

\subsection{Rigidity of full \pk immersions}\label{rig}

In the  K\"ahler case, a K\"ahler  immersion of a K\"ahler manifold into a complex space form is full if such space form has the smallest dimension. An important property of such immersions is that their restriction to any open subset of the K\"ahler  manifold is still full.
On the contrary, this property does not hold true in the \pk context. For this reason, we need to distinguish two kinds of full \pk immersions: weakly and strongly ones.
\begin{df}
A \pk immersion  $f:(M^n,\omega)\to (\s_c^N,\omega_c)$ is said to be 
\begin{itemize}
\item \emph{weakly full} if and only if $(M,\omega)$ does not admit a \pk immersion into any $ (\s_c^{N^*},\omega_c)$ with $N^*<N$;
\item \emph{strongly full} if and only if
the restriction of $f$ to any open subset of $M$ is weakly full.\end{itemize}
\end{df}
Indeed,  the restriction of  a weakly full immersion to an open subset might be not weakly full (see the example described in Section \ref{counterex}).

\begin{thm}[Rigidity]\label{rigidity}
Let $f$ and $g$ be  two weakly full \pk immersions from an open subset $U$ of a \pk manifold $(M^n,\omega)$ into $(\s_c^{N_1},\omega_c)$ and  $(\s_c^{N_2},\omega_c)$, respectively. Let also assume that
the diastasis function is defined on $U\times U$.  Then,  $N_1=N_2$ and 
$f$ differs from $g$ for a para-holomorphic isometry of the ambient space.
\end{thm}
\begin{proof}

\smallskip\noindent
We fix a para-holomorphic coordinate system 
$$z=(z_1,\mathellipsis,z_n)$$ centred at an arbitrary chosen point $p\in U$.

\smallskip
\noindent
Let firstly consider the case $c=0$.

\noindent
Being $\s_0^N$ a homogeneous space with respect to the action of its para-holomorphic isometry group, we can assume without loss generality that $f(p)=g(p)=0$.

\noindent
Since $f$ and $g$ are para-holomorphic, $f(U)$, as well as $g(U)$, cannot be  a real subspace of $\D^N$. Indeed, if $f(z)=\overline{f(z)}$ for any $z$, then $f$ would be a constant by \eqref{dezbar}. 

\noindent
We now choose $s$ points $p_1,\mathellipsis,p_s\in U$ different from $p$ and we consider $s$ real constants $\alpha_i$ such that
$$\sum_{i=1}^s \alpha_i f(p_i)=0,$$ 
from which
$$\sum_{i,j=1}^s \alpha_i\alpha_j f(p_i) \overline{ f(p_j)}=0.$$ 
In view of Proposition \ref{her}, by considering the definition of diastasis \eqref{diast} and by taking into account \eqref{dn},  we obtain that $D^{U}(p_i,p)=\|f(p_i)\|^2$. Hence, the previous equality can be written as
$$\sum_{i=1}^s \alpha_i^2 D^{U}(p_i,p)+\sum_{i<j} \alpha_i\alpha_j \left(D^{U}(p_i,p)+D^{U}(p_j,p)-D^{U}(p_i,p_j) \right) =0.$$ 
Then, it follows that the maximal number of linearly independent vectors $f(p_i)$ depends only on the diastasis function, not on the immersion.
Therefore, since both immersions $f$ and $g$ are weakly full,  the dimensions of the ambient spaces need to be equal.
Furthermore,  in view  again of Proposition \ref{her}, we have that
\begin{equation*}\label{diasteq}
D_{f(p)}^{\s_0^N}(f(q))=D_{g(p)}^{\s_0^N}(g(q)),\end{equation*}
namely
$$\|f(q)\|^2=\|g(q)\|^2, \qquad \forall q\in U\,.$$
By considering that $f(U)$ and $g(U)$ span respectively two $\D$-linear spaces having the same dimension, then such immersions differ from each other for a $\D$-unitary transformation.\\

\noindent
Now, let us consider the case $c\neq 0$.

\noindent
Being $\s_c^N$ a homogeneous space with respect to the action of its para-holomorphic isometry group, we can assume without loss generality that $f(p)=g(p)=[1:0:\mathellipsis : 0]$ and that their images are contained in the affine chart   $\{[Z_0,\mathellipsis,Z_N]\in \DP^N \ |\ Z_0\neq 0\}$. Hence, our immersions can be computed from
$$\tilde f,\tilde g: (M^n,\omega)\to \left(\D^{N+1}, \frac{4\tau}{c}\de\bar\de\log \|Z\|^2 \right)$$ 
by considering the canonical projection.
More precisely, in view of Proposition \ref{her}, by taking into account the definition of diastasis \eqref{diast} and  \eqref{dpn}, we have that 
$$D^{U}(p,q) =\frac{4}{c}\log\left( \frac{\|\tilde f(p)\|^2 \ \|\tilde f(q)\|^2}{\big|\langle \tilde f(p),\tilde f(q) \rangle\big|^2}\right).$$
Our statement follows by considering very similar arguments we adopted (above) in the  proof for the flat case.
\end{proof}

In Section \ref{counterex}, it is described an example of \pk manifold that better illustrates the local character of \pk immersions. To understand this example, one needs only Proposition \ref{existence} below.

\subsection{Characterization of full \pk immersions}\label{pk immersion}

To start with, below we give another application  of Proposition \ref{her}, which will be useful for characterizing weakly full \pk immersions into a \pk space form of para-holomorphic curvature $c\in \R$, through a suitable function $H_c$ of the diastasis.

\smallskip\noindent
If $$f: (M^n,\omega)\to (\s_c^N,\omega_c), \qquad n\leq N,$$
is a \pk immersion reading locally as (cfr. Remark \ref{ph1} and Definition \ref{ph2})
\begin{equation}\label{immersion}
f(\xi_1e+\eta_1\bar e,\mathellipsis,\xi_ne+\eta_n\bar e)=\left(u_1(\xi) e+v_1(\eta)\bar e,\mathellipsis, u_N(\xi)e+v_N(\eta)\bar e \right),\end{equation}
where $(\xi,\eta)=(\xi_1,\mathellipsis,\xi_n,\eta_1,\mathellipsis,\eta_n)$ and $(u,v)=(u_1,\mathellipsis,u_N,v_1,\mathellipsis,v_N)$ are null-coordinates on 
 an open subset
$U=\Omega\times\Omega\subseteq M^n$ and $\s_c^N$, respectively, then we have
$$D_0^{M^n}(\xi e+\eta \bar e)=D_{f(0)}^{\s_c^N}(ue+v\bar e).$$
Since, up to change of coordinates, we can assume that $f(0)=0$, i.e., in coordinates,
\begin{equation}\label{u0v0}
u(0)=0, \qquad v(0)=0,\end{equation} then
 we get, in view of \eqref{dn} and \eqref{dpn},    that
\begin{equation}
D_0^{M^n}(\xi e+\eta \bar e)=
\begin{cases}
4\sum_{i=1}^N u_i(\xi)v_i(\eta) &\text{if } c=0;\\
\frac{8}{c}\log\left(1+2\sum_{i=1}^N u_i(\xi)v_i(\eta)\right) &\text{if } c\neq 0.\\
\end{cases}
\end{equation}
Therefore, the function $H_c :U=\Omega\times\Omega\subseteq M^n\to \R$ defined by
\begin{equation}\label{H}
H_c(\xi,\eta):=
\begin{cases}
\frac{1}{4}D_0^{M^n}(\xi e+\eta \bar e)&\text{if } c=0;\\
\frac{1}{2}\exp\left( \frac{c}{8}D_0^{M^n}(\xi e+\eta \bar e)\right)-\frac{1}{2} &\text{if } c\neq 0,\\
\end{cases}
\end{equation}
reads as
\begin{equation}\label{problem}
H_c(\xi,\eta)=\sum_{\alpha=1}^N u_\alpha(\xi)v_\alpha(\eta).
\end{equation}

Vice versa, given a \pk manifold $(M^n,\omega)$, an arbitrary point $p\in M$ and a system of null-coordinates $(\xi,\eta)$ centered at $p$, whenever the condition \eqref{problem} holds true, with $H_c$  defined by \eqref{H}, there exists  a neighborhood of $p$ that can be \pk immersed into $(\s_c^N,\omega_c)$ via a \pk immersion reading as \eqref{immersion}. In fact, being the diastasis function $D_0$ a \pk potential (see Remark \ref{D0}), we notice that 
$$\det\left(\frac{\de^2 H_c}{\de \xi_i\de\eta_j}(0,0)\right)_{1\leq i,j\leq n}\neq 0.$$ Therefore,  by taking into account \eqref{H}, we get that  the Jacobians of both $(u_1,\mathellipsis,u_N)$ and $(v_1,\mathellipsis,v_N)$  need to have maximal rank in a neighborhood of $p$.
To sum up we have the following  proposition.
\begin{prop}\label{existence}
Let $U$ be an open subset of \pk manifold $(M^n,\omega)$ where it is defined a system of null-coordinates $(\xi,\eta)$. Then, there exists a weakly full \pk immersion
 $f: (U,\omega|_U)\to (\s_c^N,\omega_c)$
if and only if the function $H_c$ defined by \eqref{H} reads as \eqref{problem} with the smallest possible $N$. 
\end{prop}
Below, in Theorem  \ref{main}, we provide some more practical criteria to verify if the conditions of the previous proposition hold true. To this aim, we need to  introduce the following definition.
\begin{df}
A \emph{generalized Wronskian} of $u_1(\xi),\mathellipsis,u_N(\xi)$ is any determinant of the type
$$\det\begin{pmatrix}
\mathcal{D}_0(u_1)&\dots&\mathcal{D}_0(u_N) \\
\vdots & &\vdots\\
\mathcal{D}_{N-1}(u_1)&\dots&\mathcal{D}_{N-1}(u_N)\\
\end{pmatrix},$$
where $\mathcal{D}_j$ denotes a partial derivative $\frac{\de^{|I|}}{\de\xi^I}$ where $|I|\leq j$. In particular, $\mathcal{D}_0=\mathrm{id}$. A \emph{generalized $r\times r$ sub-Wronskian} of $u_1,\mathellipsis,u_N$ is a Wronskian of a subset  of $u_1,\mathellipsis,u_N$ containing $r$ elements. A point $p$ where all the functions $\{u_1,\mathellipsis,u_N\}$ are defined,
is said to be of \emph{order} $\ord(p)\in\N$, if any $r\times r$ sub-Wronskian vanishes at $p$ for  $r>\ord(p)$.
\end{df}

\begin{thm}\label{main}
Let $p$ be a point of a \pk manifold $(M^n,\omega)$. Let $U=\Omega\times\Omega$ be a neighborhood of $p$ where it is defined a system of null-coordinates $(\xi,\eta)$ centered at $p$. Let the diastasis function be defined on $U\times U$.
Then  necessary conditions for the existence of a  weakly  full \pk immersion 
$$f: (U,\omega|_U)\to (\s_c^N,\omega_c)$$
are the following:\\
There exist two sets  $\I,\J\subset \N^n$ of multi-indices, containing 0 and having finite cardinality, for which $\forall\, K\in\N^n$ there exist
smooth functions $a^I_K, b^J_K:\Omega \to\R $, where $I\in\I$, $J\in\J$, $K\in\N^n$ and two  smooth functions 
$a, b$ not identically zero in a neighborhood  of $0$ with $a(0)=b(0)=0$, 
such that
\begin{equation}\label{dip1}
a(\xi)\ \frac{\de^{|K|}H_c}{{\de\xi^{K}}}(\xi,\eta) + \sum_{I\in\I} a^I_K(\xi)\   \frac{\de^{|I|}H_c}{\de\xi^{I}} (\xi,\eta) \equiv 0 
\end{equation}
and
\begin{equation}\label{dip2}
b(\eta)\ \frac{\de^{|K|}H_c}{{\de\eta^{K}}}(\xi,\eta) + \sum_{J\in\J} b^J_K(\eta)\   \frac{\de^{|J|}H_c}{\de\eta^{J}} (\xi,\eta) \equiv 0
\end{equation}
where the function $H_c$ is defined by \eqref{H}.  The above conditions are sufficient by assuming that either \begin{equation}\label{eq:A}
\Omega'=\big\{\xi\in\Omega\ | \ a(\xi)\neq 0\big\}\,,
\end{equation}
or $\big\{\eta\in\Omega\ | \ b(\eta)\neq 0\big\}$,
is a connected  open dense subset of $\Omega$. This leads actually to a strongly full \pk immersion.
\end{thm}

\begin{proof}
By differentiating \eqref{problem}, we get
\begin{equation}\label{s1}
\begin{pmatrix}
\mathcal{D}_0(u_1)&\dots&\mathcal{D}_0(u_N)\\
\vdots & &\vdots\\
\mathcal{D}_{N-1}(u_1)&\dots&\mathcal{D}_{N-1}(u_N)\\
\mathcal{D}_{j}(u_1)&\dots&\mathcal{D}_{j}(u_N)\\
\end{pmatrix}\begin{pmatrix}
 v_1\\
\vdots\\
v_N\end{pmatrix}=\begin{pmatrix}
\mathcal{D}_0(H_c)\\
\vdots\\
\mathcal{D}_{N-1}(H_c)\\
\mathcal{D}_{j}(H_c)
\end{pmatrix}\end{equation}
where the subscript $j$ is arbitrary chosen in $\N$.
Since \eqref{s1} can be seen as a linear system in $(v_1,\mathellipsis,v_N)$, it follows, from its compatibility condition, that
\begin{equation}\label{detnec}
\det\begin{pmatrix}
\mathcal{D}_0(u_1)&\dots&\mathcal{D}_0(u_N) &\mathcal{D}_0(H_c)\\
\vdots & &\vdots&\vdots\\
\mathcal{D}_{N-1}(u_1)&\dots&\mathcal{D}_{N-1}(u_N)&\mathcal{D}_{N-1}(H_c)\\
\dd_j(u_1)&\dots& \dd_j(u_N)&\dd_j(H_c)\\
\end{pmatrix}= 0\qquad \forall (\xi,\eta)\in\Omega\times\Omega.\end{equation}
We notice that the first row of \eqref{detnec} is equal to $0$ at $\xi=0$  by \eqref{u0v0} and by definition of $H_c$.

\smallskip
\noindent
We have now to consider two cases: the first is when there exists a generalized Wronskian which does not vanish identically on $\Omega$ and the second when all generalized Wronskians vanish identically on $\Omega$.

\smallskip
\noindent
\textbf{The case  when there exists a generalized Wronskian which does not vanish identically on $\Omega$.}  In this case, the condition \eqref{dip1} follows straightforwardly from the Laplace's expansion of \eqref{detnec}  w.r.t. the last column, by suitably choosing the operators $\mathcal{D}_0,\mathellipsis,\mathcal{D}_{N-1}$ so that
the corresponding Wronskian of $u_1(\xi),\mathellipsis,u_N(\xi)$
does not vanish identically  on  $\Omega$.

\smallskip
\noindent
\textbf{The case  when all generalized Wronskians vanish identically on $\Omega$.} Firstly, we notice that we cannot have open subsets of $\Omega$ consisting only of zero-order points.
Indeed, if we assume the existence of a subset $\tilde\Omega\subseteq \Omega$ of such type, all the restrictions of the function $u_1,\mathellipsis, u_N$ to $\tilde\Omega$ would be identically zero. Being these functions the components of an immersion, we would get a contradiction.

\smallskip\noindent
Now, we pick a point $p\in\Omega$ of maximal order. Hence, up to renaming the functions $u_i$, we have
$$
\det \begin{pmatrix}
\mathcal{D}_0(u_1)&\dots&\mathcal{D}_0(u_{\ord(p)})\\
\vdots & &\vdots\\
\mathcal{D}_{\ord(p)-1}(u_1)&\dots&\mathcal{D}_{\ord(p)-1}(u_{\ord(p)})\\
\end{pmatrix}(p)\neq 0.$$
 By taking into account the system 
$$\begin{pmatrix}
\mathcal{D}_0(u_1)&\dots&\mathcal{D}_0(u_N)\\
\vdots & &\vdots\\
\mathcal{D}_{\ord(p)-1}(u_1)&\dots&\mathcal{D}_{\ord(p)-1}(u_N)\\
\mathcal{D}_{j}(u_1)&\dots&\mathcal{D}_{j}(u_N)\\
\end{pmatrix}\begin{pmatrix}
 v_1\\
\vdots\\
v_N\end{pmatrix}=\begin{pmatrix}
\mathcal{D}_0(H_c)\\
\vdots\\
\mathcal{D}_{\ord(p)-1}(H_c)\\
\mathcal{D}_{j}(H_c)
\end{pmatrix},$$
where the subscript $j$ is  arbitrary chosen in $\N$,
we have, up to renaming $u_i$, that
$$\det\begin{pmatrix}
\mathcal{D}_0(u_1)&\dots&\mathcal{D}_0(u_{\ord(p)}) &\mathcal{D}_0(H_c)\\
\vdots & &\vdots&\vdots\\
\mathcal{D}_{\ord(p)-1}(u_1)&\dots&\mathcal{D}_{\ord(p)-1}(u_{\ord(p)})&\mathcal{D}_{\ord(p)-1}(H_c)\\
\dd_j(u_1)&\dots& \dd_j(u_{\ord(p)})&\dd_j(H_c)\\
\end{pmatrix}= 0\qquad \forall (\xi,\eta)\in\Omega\times\Omega.$$
The condition \eqref{dip1} follows from the Laplace's expansion w.r.t. the last column of the previous determinant.\\
With a similar reasoning, we get \eqref{dip2}.

\medskip\noindent
Now we prove that condition stated in the Theorem are also sufficient under the hypothesis that $\Omega'$, given by \eqref{eq:A}, be a connected open dense subset of $\Omega$.
This will be done in several steps.

\bigskip\noindent
\textbf{\large Step 1.} To start with, in this step, we define a distinguished set of multi-indices (that we are going to call $\I$ with slight abuse of notation), whose cardinality is strictly related to the dimension of the ambient space form. 
\smallskip\noindent
 Let  $\mathcal{I}=\bigcup_{i\in\N}\mathcal{I}_i\subseteq\N^n$, where
\begin{itemize}
\item $\mathcal{I}_0=\{I_0\}$,
\item $\mathcal{I}_{i+1}=\mathcal{I}_i\cup\{I_{i+1}\}$ 
 if   
$$ \sum_{I\in\mathcal{I}_i\cup\{I_{i+1}\}}c_I(\xi)\frac{\de^{|I|} H_c}{\de\xi^I}(\xi,\eta)\equiv 0$$
holds true if and only if all the $c_I\in \mathcal{C}^\infty(\Omega)$ vanish identically on $\Omega$. 
Otherwise $\mathcal{I}_{i+1}=\mathcal{I}_i$.
\end{itemize}
Therefore, in view of the conditions \eqref{dip1} and by construction, $\#\I\in\N$. Let$$N=\#\I.$$
Below we are going to show the existence of a  strongly full  \pk immersion  
$$f: (M^n,\omega)\to (\s_c^{N},\omega_c).$$

\medskip\noindent
\textbf{\large Step 2.} In this step, we will study a particular system of first order PDEs whose compatibility is strictly related to the existence of a \pk immersion \eqref{immersion}.

\smallskip\noindent
We define $e_k\in\N^n$ as follows:
$$e_k=(0,\mathellipsis,0,1,0\mathellipsis,0),\quad \text{where $1$ is at the $k$-th position}.$$
The condition \eqref{dip1} reads in particular as 
\begin{equation}\label{dip}
a(\xi)\ \frac{\de^{|I|+1}H_c}{{\de\xi^{I+e_k}}}(\xi,\eta) + \sum_{J\in\mathcal{I}} a^J_{I+e_k}(\xi)\   \frac{\de^{|J|}H_c}{\de\xi^{J}} (\xi,\eta) \equiv 0,\end{equation}
for any $I\in\I$ and  for   $1\leq k \leq n$.
The functions 
$$
A_k{}^J_I:=\frac{a^J_{I+e_k}}{a}
$$
are defined on the open non-empty subset $\Omega'$, see \eqref{eq:A}.
Now we are going to study the following system of first order PDEs,
\begin{equation}\label{sistema}
\frac{\de U_{IJ}}{\de\xi_k}-\sum_{L\in\I}A_k{}^L_J \, U_{IL}=0\,,\quad k=1,\dots,n\,,
\end{equation}
where $U_{IJ}$ are $N^2$ unknown functions defined on $\Omega'$, whose solutions will allow us to prove that $H_c$ reads as \eqref{problem}.  This will be done  in Step 3.

\noindent
By recalling the definition \eqref{iota} of $\iota$, we can define the $N\times N$ matrices $U$ and $A_k$, for any $k\in\{1,\dots,n\}$, as the matrices whose entries are, respectively, $U_{\iota(I)\iota(J)}$ and $A_k{}^{\iota(I)}_{\iota(J)}$, where $(I,J)\in\I\times\I$:
\begin{equation}\label{U.Ak}
U=\left(U_{\iota(I)\iota(J)}\right)\,,\quad A_k=\left(A_k{}^{\iota(I)}_{\iota(J)}\right)\,.
\end{equation}
Keeping in mind this notation, system \eqref{sistema} assumes the more concise form:
\begin{equation}\label{sistema2}
\frac{\de U }{\de\xi_k }-U\cdot A_k=0\,,\quad k=1,\dots,n\,.
\end{equation}
We can successfully study the above system thanks to the following lemma, that is a classical result coming from the Frobenius Theorem on the complete integrability of vector distributions \cite{warner}.
\begin{lemma}\label{lemma.Frob}
Let $\xi=(\xi_1,\dots,\xi_n)$ be coordinates on a connected open subset $\mathcal{W}$ and $V=(V_1,\dots,V_m)$, where $V_i\in\mathcal{C}^\infty(\mathcal{W})$, be a row vector. The following Cauchy problem
\begin{equation}\label{eq:sistema.V}
\left\{
\begin{array}{ll}
\frac{\de V }{\de\xi_k }+ V\cdot B_k=0 &
\\
\\
V(\xi_0)=V_{0}&
\end{array}
\right.
\end{equation}
where $B_k=\big(B_k{}^i_j\big)$ with
$B_k{}^i_j\in\mathcal{C}^\infty(\mathcal{W})$ and
$\xi_0\in\mathcal{W}$, such that the compatibility conditions
\begin{equation}\label{eq:cond.comp.5}
\frac{\de B_h }{\de\xi_k }+B_h\cdot B_k-\frac{\de B_k }{\de\xi_h }-B_k\cdot B_h=0\,,\quad\forall\,k,h\,,
\end{equation}
are satisfied, admits a unique solution defined on $\mathcal{W}$.
\end{lemma}
\begin{proof}
We only notice that \eqref{eq:sistema.V} is a Pfaffian system of $m$ unknown functions $V_i$ and $n$ independent variables $\xi_k$, i.e., any solution $V=V(\xi)$ annihilates the $m$ differential 1-forms
$$
\rho_i= d V_i + \sum_{a=1}^m\sum_{k=1}^n B_k{}^a_i V_a\, d\xi_k\,, \quad i=1,\dots, m\,.
$$
The system $\{\rho_i=0\}_{i=1,\dots,m}$ has constant maximal rank equal to $m$ and, since its (Frobenius) compatibility conditions
$$
d\rho_i=\sum_{j=1}^m\rho_{ij}\wedge\rho_j\,,\quad i=1,\dots,m
$$
where $\rho_{ij}$ are suitable $1$-differential forms,
are satisfied if and only if \eqref{eq:cond.comp.5} holds, the assertion of the lemma follows.
\end{proof}
Now, going back to system \eqref{sistema}, since it is of type \eqref{eq:sistema.V} (with some obvious substitutions), in order to apply Lemma \ref{lemma.Frob} to such system,
we shall show that its compatibility conditions are satisfied. To this aim, in view of \eqref{dip}, we have
\begin{multline}\label{1}
\frac{\de}{\de \xi_h}     \frac{\de^{|I|+1} H }{\de\xi^{I+e_k} }= 
- \sum_{J\in\mathcal{I}} \frac{\de A_k{}^J_I }{\de\xi_h }
 \,  \frac{\de^{|J|} H }{\de\xi^J }-  \sum_{J\in\mathcal{I}} A_k{}^J_I\,    \frac{\de^{|J|+1} H }{\de\xi^{J+e_h} }
= \sum_{J\in\mathcal{I}}\left( -  \frac{\de  A_k{}^J_I}{\de \xi_h } +  \sum_{L\in\mathcal{I}}A_k{}^L_I A_h{}^J_L\right)   \frac{\de^{|J|} H }{\de\xi^J }
\end{multline}
and
\begin{equation}\label{2}
\frac{\de}{\de \xi_k}     \frac{\de^{|I|+1} H }{\de\xi^{I+e_h} }  = 
 \sum_{J\in\mathcal{I}}\left( -  \frac{\de A_h{}^J_I  }{\de \xi_k} +  \sum_{L\in\mathcal{I}}A_h{}^L_I A_k{}^J_L\right)  \frac{\de^{|J|} H }{\de\xi^J }\,.
\end{equation}
By subtracting \eqref{2} to \eqref{1} and by taking into account the construction of $\I$, we get
\begin{equation}\label{comp}
\frac{\de A_h{}^J_I }{\de \xi_k} - \frac{\de A_k{}^J_I }{\de \xi_h} +  \sum_{L\in\mathcal{I}} A_h{}^J_LA_k{}^L_I  -  \sum_{L\in\mathcal{I}} A_k{}^J_L A_h{}^L_I =0\,,
\end{equation}
that are exactly the compatibility conditions of system \eqref{sistema}.

\noindent
Let us adopt the notation \eqref{U.Ak}-\eqref{sistema2}.
We underline that if a solution $U=U(\xi)$ of system \eqref{sistema2} is such that $\det(U(\xi'))\neq 0$ for some $\xi'\in\Omega'$ (see \eqref{eq:A}), then $\det(U(\xi))\neq 0$ for all $\xi\in\Omega'$.
 This follows from the following consideration. If, by contradiction, there existed $\xi_0\in \Omega'$ such that $\det(U(\xi_0))=0$, then it would exist a (row) vector $v\neq 0$ such that $v\cdot U(\xi_0)=0$. Therefore, the function $v\cdot U:\xi\mapsto v\cdot U(\xi)$ would be a solution of the following Cauchy problem
$$
\left\{
\begin{array}{l}
\frac{\de }{\de \xi_k}(v\cdot U)-(v\cdot U)\cdot A_k=0
\\
\\
(v\cdot U)(\xi_0)=0
\end{array}
\right.
$$
implying, in view of Lemma \ref{lemma.Frob}, $(v\cdot U)(\xi)=0$ for all $\xi\in\Omega'$, in particular $(v\cdot U)(\xi')=v\cdot U(\xi')=0$, leading to $\det(U(\xi'))=0$, a contradiction.

%
%

\medskip\noindent
\textbf{\large Step 3.}
In this step, we will show that the equality  \eqref{problem} holds true on $\Omega'\times\Omega$.

\smallskip\noindent
 To this aim,  we fix  $\xi_0\in\Omega'$ and a matrix $U_0\in \R^{N,N}$ with $\det U_0\neq 0$.
Taking into account the last part of Step 2 and Lemma \ref{lemma.Frob}, there exists a unique solution $U=U(\xi)$, defined on  $\Omega'$, to system \eqref{sistema2} such that $U(\xi_0)=U_0$.
By recalling the definition \eqref{iota} of $\iota$, if 
\begin{equation}\label{v}
v_{\iota(I)}=\sum_{J\in\I} U_{IJ}\,  \frac{\de^{|J|} H }{\de\xi^J }, 
\end{equation}
then
$\frac{\de v_{\iota(I)}}{\de \xi_k}\equiv 0$ for any $I\in\I$ and for any $1\leq k\leq n$. In fact, by \eqref{dip} and \eqref{sistema}
$$
\frac{\de v_{\iota(I)}}{\de \xi_k}=\sum_{J\in\I} \frac{\de U_{IJ} }{\de\xi_k }\,\frac{\de^{|J|} H }{\de\xi^J } + \sum_{J\in\I}U_{IJ}\, \frac{\de^{|J+1|}H}{\de\xi^{J+e_k}}
= 
\sum_{J\in\I} \sum_{L\in\I} A_k{}^L_J \,U_{IL}\, \frac{\de^{|J|} H }{\de\xi^J } -\sum_{J\in\I} \sum_{L\in\I}A_k{}^L_J \,U_{IL}\frac{\de^{|J|} H }{\de\xi^J }=0\,.
$$
Therefore, $v_{\iota(I)}=v_{\iota(I)}(\eta)\in\mathcal{C}^\infty(\Omega)$ for any $I\in\I$. {Furthermore, we notice that, by definition of $H_c$, we have $\frac{\de^{|J|} H_c}{\de \xi^J}(\xi,0)=0$ for any $J\in\N^n$ and for any $\xi\in\Omega$. Hence, in view of \eqref{v}, we have
$$v_i(0)=0,\qquad \forall i=1\mathellipsis,N.$$}
Keeping in mind the notation \eqref{U.Ak} and let
\begin{equation}\label{u}
u(\xi)=(1,0,\mathellipsis,0)\cdot  U^{-1}\end{equation}
where $(1,0,\mathellipsis,0)$ is a $1\times N$ matrix, we have, by construction,
\begin{equation}\label{eq}
\sum_{i=1}^N u_i(\xi)v_i(\eta)=H(\xi,\eta),\qquad \forall(\xi,\eta)\in  \Omega'\times \Omega.\end{equation}

\medskip\noindent
\textbf{\large Step 4.} 
In this step, we are going to extend \eqref{eq} to $\Omega\times\Omega$ when the linear closure of $(v_1,\mathellipsis,v_N)\left(\Omega\right)$ has dimension $N$. 
Under such assumption, there exist some $\eta_1,\mathellipsis,\eta_N\in\Omega$ such that
\begin{equation*}\det
 \begin{pmatrix}
v_1(\eta_1)&\dots&v_1(\eta_N)\\
\vdots & &\vdots\\
 v_N(\eta_1)&\dots& v_N(\eta_N)
\end{pmatrix}\neq 0.
\end{equation*}
By considering \eqref{eq}, we have
$$ 
(u_1,\mathellipsis,u_N)\,(\xi)=\big(H_c(\xi,\eta_1),\mathellipsis,H_c(\xi,\eta_N)\big)
 \begin{pmatrix}
v_1(\eta_1)&\dots&v_1(\eta_N)\\
\vdots & &\vdots\\
 v_N(\eta_1)&\dots& v_N(\eta_N)
\end{pmatrix}^{-1}, \qquad \forall\xi\in\Omega'.
$$
Therefore, $u_1,\mathellipsis,u_N$ admit a smooth extension to $\Omega$. Thus, we have that \eqref{eq} holds true on the whole $\Omega\times\Omega$. Moreover, since $H_c(0,\eta)\equiv 0$ by construction, we have that $(u_1,\mathellipsis,u_N)(0)=0$.

\medskip\noindent
\textbf{\large Step 5.} 
Finally, we  show that  the linear closure of $(v_1,\mathellipsis,v_N)\left(\Omega\right)$ cannot have dimension  less than $N$. Indeed, if we assume by contradiction  the  existence of some $\eta_1,\mathellipsis,\eta_\rho\in\Omega$ such that
\begin{equation}	\label{rank1}
\mathrm{rank} \begin{pmatrix}
v_1(\eta_1)&\dots&v_1(\eta_\rho)\\
\vdots & &\vdots\\
 v_N(\eta_1)&\dots& v_N(\eta_\rho)
\end{pmatrix}= \rho
\end{equation}
 and 
\begin{equation}	\label{rank2}
\mathrm{rank} \begin{pmatrix}
v_1(\eta_1)&\dots&v_1(\eta_\rho)&v_1(\eta)\\
\vdots & &\vdots&\vdots\\
 v_N(\eta_1)&\dots& v_N(\eta_\rho)&v_N(\eta)\\
\end{pmatrix}= \rho,\qquad \forall \eta\in\Omega,\end{equation}
then, by considering the construction \eqref{v} of the functions $v$ (recall that $U$ is an invertible matrix) and  by taking into account \eqref{rank2},   we have that
$$\mathrm{rank} \begin{pmatrix}
\frac{\de^{| I_{k_1}| }H_c}{\de \xi^{I_{k_1}}}(\xi,\eta_1)&\dots& \frac{\de^{| I_{k_1}| }H_c}{\de \xi^{I_{k_1}}}(\xi,\eta_\rho)&\frac{\de^{|I_{k_1} | }H_c}{\de \xi^{I_{k_1}}}(\xi,\eta)\\
\vdots & &\vdots&\vdots\\
\frac{\de^{|I_{k_N} | }H_c}{\de \xi^{I_{k_N}}}(\xi,\eta_1)&\dots& \frac{\de^{|I_{k_N} | }H_c}{\de \xi^{I_{k_N}}}(\xi,\eta_\rho)&\frac{\de^{| I_{k_N}| }H_c}{\de \xi^{I_{k_N}}}(\xi,\eta)\\
\end{pmatrix}= \rho,\qquad \forall \eta\in\Omega,\ \forall \xi\in\Omega',$$
where $\I=\{I_{k_1},\mathellipsis,I_{k_N}\}$. It follows that
$$\det \begin{pmatrix}
\frac{\de^{| I_{k_1}| }H_c}{\de \xi^{I_{k_1}}}(\xi,\eta_1)&\dots& \frac{\de^{| I_{k_1}| }H_c}{\de \xi^{I_{k_1}}}(\xi,\eta_\rho)&\frac{\de^{|I_{k_1} | }H_c}{\de \xi^{I_{k_1}}}(\xi,\eta)\\
\vdots & &\vdots&\vdots\\
\frac{\de^{|I_{k_{\rho+1}} | }H_c}{\de \xi^{I_{k_{\rho+1}}}}(\xi,\eta_1)&\dots& \frac{\de^{|I_{k_{\rho+1}} | }H_c}{\de \xi^{I_{k_{\rho+1}}}}(\xi,\eta_\rho)&\frac{\de^{| I_{k_{\rho+1}}| }H_c}{\de \xi^{I_{k_{\rho+1}}}}(\xi,\eta)\\
\end{pmatrix}= 0,\qquad \forall \eta\in\Omega,\ \forall \xi\in\Omega'.$$
By the Laplace's expansion of the previous determinant w.r.t the last column, we get
\begin{equation}\label{restr.eq} \sum_{i=1}^{\rho+1} c_i(\xi) \frac{\de^{|I_{k_i} | }H_c}{\de \xi^{I_{k_i}}}(\xi,\eta)=0, \qquad \forall (\xi,\eta)\in\Omega'\times\Omega,\end{equation}
where $$c_i:\Omega\to\R$$ denotes the cofactor correspondent to $i$-th entry of the last column. 
By taking into account \eqref{rank1} and \eqref{v}, not all $c_i$ can be identically zero. 
Since we are assuming that all the indices $I_{k_i}$ belong to $\I$, we get a contradiction to the construction of $\I$, cfr. Step 1.

\medskip\noindent
\textbf{\large Step 6.} 
We notice that the \pk immersion obtained by means the previous steps is strongly full. In fact, with a reasoning similar the one adopted in Step 5, we can prove that the dimension of the linear closure $(v_1,\mathellipsis,v_N)(\tilde\Omega)$ is $N$ for any open subset $\tilde\Omega$ of $\Omega$. Moreover, also the dimension of the linear closure $(u_1,\mathellipsis,u_N)(\tilde\Omega)$ is $N$ for any open subset $\tilde\Omega$ of $\Omega$. Indeed, by differentiating \eqref{problem} w.r.t. $\xi^{I_k}$ for any $I_k\in\I$, we get a linear system whose compatibility condition reads as
$$\mathrm{rank}\begin{pmatrix}
\frac{\de^{| I_{k_1}| }u_1}{\de \xi^{I_{k_1}}}&\dots& \frac{\de^{| I_{k_1}| }u_N}{\de \xi^{I_{k_1}}}\\
\vdots & &\vdots\\
\frac{\de^{|I_{k_N} | }u_1}{\de \xi^{I_{k_N}}}&\dots& \frac{\de^{|I_{k_N} | }u_N}{\de \xi^{I_{k_N}}}\\
\end{pmatrix}=\mathrm{rank}\begin{pmatrix}
\frac{\de^{| I_{k_1}| }u_1}{\de \xi^{I_{k_1}}}&\dots& \frac{\de^{| I_{k_1}| }u_N}{\de \xi^{I_{k_1}}}&\frac{\de^{|I_{k_1} | }H_c}{\de \xi^{I_{k_1}}}\\
\vdots & &\vdots&\vdots\\
\frac{\de^{|I_{k_N} | }u_1}{\de \xi^{I_{k_N}}}&\dots& \frac{\de^{|I_{k_N} | }u_N}{\de \xi^{I_{k_N}}}&\frac{\de^{| I_{k_N}| }H_c}{\de \xi^{I_{k_N}}}\\
\end{pmatrix} \qquad \text{on } \Omega\times\Omega.$$
Hence, if the left side of the previous equality is at most equal to $\rho<N$ on $\tilde\Omega\subset\Omega$, then we have
$$\det \begin{pmatrix}
\frac{\de^{| I_{k_1}| }u_1}{\de \xi^{I_{k_1}}}&\dots& \frac{\de^{| I_{k_1}| }u_N}{\de \xi^{I_{k_1}}}&\frac{\de^{|I_{k_1} | }H_c}{\de \xi^{I_{k_1}}}\\
\vdots & &\vdots&\vdots\\
\frac{\de^{|I_{k_{\rho+1}} | }u_1}{\de \xi^{I_{k_{\rho+1}}}}&\dots& \frac{\de^{|I_{k_{\rho+1}} | }u_N}{\de \xi^{I_{k_{\rho+1}}}}&\frac{\de^{| I_{k_{\rho+1}}| }H_c}{\de \xi^{I_{k_{\rho+1}}}}\\
\end{pmatrix}= 0,\qquad \forall \eta\in\Omega,\ \forall \xi\in\tilde \Omega.$$
Since we can assume w.l.g. that not all the $(\rho\times\rho)$-minors of the previous matrix vanish identically on $\tilde\Omega$, the Laplace's expansion w.r.t. the last column leads us to a contradiction to the construction of $\I$, cfr. Step 1.
\end{proof}

To conclude, we provide an example of some \pk metrics satisfying the  connectedness and density conditions of \eqref{eq:A}.
\begin{ex}
Let $(\xi,\eta)=(\xi_1,\mathellipsis ,\xi_n,\eta_1,\mathellipsis ,\eta_n)$ be a  system of null-coordinates coordinates on an $n$-dimensional \pk manifold, such that the restriction $D_0(\xi,\eta)$ of its diastasis function to a  neighborhood $\Omega\times \Omega$ of the center of the coordinate system, is related (cfr. \eqref{H}) to the function $H_c$ reading as
$$H_c(\xi,\eta)=\sum_{i=1}^n\xi_i\eta_i+\left(\xi_1^2+\mathellipsis+\xi_n^2\right)^h\left(\eta_1^2+\mathellipsis+\eta_n^2\right)^k,$$
where $h$ and $k$ denote two arbitrary positive integer numbers. Via straightforward computations, we get that
$$\left(\eta_1^2+\mathellipsis+\eta_n^2\right)^k \frac{\de^{|I|}H_c}{\de \eta^I}+ a_I(\eta) \left( \sum_{i=1}^n \eta_i \frac{\de H_c}{\de \eta_i}-H_c\right)=0$$
holds true for any $I\in\N^n$ such that $|I|\geq 2$ (after suitably choosing the smooth functions $a_I:\Omega\to\R$). Therefore, the set $\Omega'$, defined by \eqref{eq:A}, can be obtained from $\Omega$ by removing only the projection on $\Omega$ of the center of the coordinates system.
\end{ex}

\subsection{On the local character and global extendability of \pk immersions}\label{counterex}

As proved by Calabi in \cite{Cal}, if an open subset of a connected K\"ahler manifold can be holomorphically  and isometrically immersed (i.e.  K\"ahler immersed) into a complex space form, then any point of such manifold admits a neighborhood that can be K\"ahler immersed in the same ambient space. Moreover, any  K\"ahler immersion into a complex space form defined on an open subset of  a  simply connected K\"ahler manifold can be extended to an immersion defined on the whole manifold.\\
As the following example shows, such global aspects of local K\"ahler  immersions are not shared with their \pk counterparts.

\smallskip

Indeed, below we are going to show that there exists a \pk structure $\omega$ on $\D=\D^1$ for which, for any point $z_0\in\D$, there exist neighborhoods of $z_0$  covering $\D$, such that they are weakly full \pk immersed in flat \pk space forms of different dimensions. As a consequence, it turns out that $(\D,\omega)$ cannot be globally \pk immersed into any flat \pk space form.

\smallskip

Let $i\in \N$ and let $u_i:\R\to\R$ be the smooth (non-analytic in $i$) function defined by
\begin{equation}
u_i(\xi)=\begin{cases}
\exp\left( -\frac{1}{(\xi-i)^{i+1}}\right) & \text{if } \xi> i\\
0 & \text{if } \xi\leq i
\end{cases}
\end{equation}
We consider the one dimensional para-complex manifold $\D$ endowed with the symplectic form associated to the (global) potential
\begin{equation}\begin{array}{rccl}
\Phi:&\D&\longrightarrow&\R\\
&(\xi e+\eta \bar e)&\longmapsto& \xi \eta+\sum_{i=0}^{+\infty} u_i(\xi) \eta^{2i+3},\\
\end{array}
\end{equation}
namely, the 2-form reading in null-coordinates $(\xi,\eta)$ as
$$\omega=\frac{1}{2}\left(1 +\sum_{i=0}^{+\infty} (2i+3)u_i'(\xi) \eta^{2i+2}\right) d\xi\wedge d\eta,$$
where
$$u_i'(\xi)=\begin{cases}
\frac{\exp\left( -\frac{1}{(\xi-i)^{i+1}}\right)}{(\xi-i)^{i+2}} & \text{if } \xi> i\\
0 & \text{if } \xi\leq i
\end{cases}$$
Let
$$\mathcal{U}_i:=\{\xi e+\eta \bar e\in \D\ |\  (\xi,\eta)\in\left(-\infty, i+1 \right)\times \R\}, \quad i\in\N\,.$$
Since 
$$\Phi\big|_{\mathcal{U}_i}(\xi,\eta)= \xi \eta+\sum_{j=0}^{i} u_j(\xi) \eta^{2j+3}, $$ 
in view of \eqref{H}, there exists   a weakly full \pk immersion of $\left(\mathcal{U}_i,\omega\big|_{\mathcal{U}_i}\right)$ into $(\D^{i+1},\omega_0)$. \\

Furthermore, by considering that $\D=\cup_{i\in\N} \ \mathcal{U}_i$, we conclude  that $(\D,\omega)$ cannot be globally \pk immersed into any flat \pk space form.
%
%
%
%
%

\section{Para-K\"ahler space forms that admit a \pk immersion into another}\label{pkspforms}

As an application of Theorem \ref{main}, in the present section we classify  \pk immersions between \pk space forms. In particular, we are going to prove the following theorem. 
\begin{thm}
A \pk space form $(\s_c^n,\omega_c)$ can be locally \pk immersed into $(\s_b^N,\omega_b)$, where $N\geq n$, if and only if either $c=b=0$ or $b/c\in\Z^+$. Such local immersions can be extended to the whole manifold. Moreover, if
$N=n$ (when $c=b=0$), or if $N =\binom{n+\frac{b}{c}}{n}-1$ (when $\frac{b}{c}\in\Z^+$), these immersions are strongly full.
\end{thm}
\begin{proof}

Since \pk space forms are homogeneous with respect to the action of their para-homolorphic isometry group (see the end of Section \ref{sec.pk.diastasis}), the existence (or non-existence) of  a \pk immersion defined on a neighborhood of an arbitrary chosen point implies the existence (or non-existence) of a \pk immersion defined on a neighborhood of any other point. Therefore, we can study, without loss of generality, only local \pk immersions defined either on a neighborhood of  $0\in \s_0^n$ or on a neighborhood of  $[1,0\mathellipsis,0]\in\s^n_c$, $c\neq 0$. Let us fix a null coordinate system $(\xi,\eta)$ around such point. Keeping in mind the definitions \eqref{dn}, \eqref{dpn} and \eqref{H},
we compute the function $H_b$ in order to apply Theorem \ref{main}.
Notice that a flat complex space form $(\s_0^n,\omega_0)$ can be trivially \pk immersed by inclusion into $(\s_0^N,\omega_0)$, hence only the following cases will be taken into account:
\begin{itemize} 
\item $c\neq 0$ and $b= 0$;
\item $c= 0$ and $b\neq 0$;
\item $c\neq 0$ and $b\neq 0$.
\end{itemize} 
Below we treat the above cases separately.

\medskip
\noindent
$\bullet$ If $c\neq 0$ and $b= 0$, by taking into account   \eqref{H} and that the diastasis $D_0^{\s_c^n}$ is given by \eqref{dpn}, we have that
$$H_0(\xi,\eta)= \frac{2}{c}\log\left( 1+2\sum_{i=1}^n \xi_i\eta_i\right).$$
By differentiating the previous equation with respect to  $\xi_1$, we get, for any $k\in\Z^+$,
$$\frac{\de^k H_0}{\de \xi_1^{k}} = \frac{(-2)^{k+1}(k-1)!}{c \left( 1+2\sum_{i=1}^n \xi_i\eta_i\right)^k}\ \eta_1^k\,.$$
Let now assume that the condition \eqref{dip1} of Theorem \ref{main} holds true for any  $\I\subset \N^n$ such that $\#\I=N+1$. In particular, if 
\begin{equation}\label{II}
\I=\{(k,0,\mathellipsis,0)\ |\ k=1,\mathellipsis,N+1\},\end{equation}
we have that
$$\sum_{k=1}^{N+1} a_k(\xi) \frac{\de^k H_0}{\de \xi_1^{k}} (\xi,\eta)=\frac{ \sum_{k=1}^{N+1} (-2)^{k+1}(k-1)!  \left( 1+2\sum_{i=1}^n \xi_i\eta_i\right)^{N+1-k} a_k(\xi) \eta_1^k}{c \left( 1+2\sum_{i=1}^n \xi_i\eta_i\right)^{N+1}} \equiv 0.$$
%
%
%
Since the numerator in the above equality is a polynomial in $\eta_1,\mathellipsis,\eta_n$ and the coefficient of its monomial in $\eta_1$ of degree 1 is $4a_1(\xi)$, it follows that $a_1(\xi)$ needs to be identically zero.
 By means of similar considerations, we get also that $a_k(\xi)\equiv 0$ for any  $k=2,\mathellipsis,N+1$.
In view of  Theorem  \ref{main}, we  conclude that there  are no local  \pk immersions of a non-flat \pk space form into a flat one.\\

\medskip
\noindent
$\bullet$ 
If $c=0$ and $b\neq 0$, by taking into account   \eqref{H} and that the diastasis $D_0^{\s_0^n}$ is given by \eqref{dn}, we have that
$$ H_b(\xi,\eta)= \frac{1}{2}\exp\left( \frac{b}{2}\sum_{i=1}^n \xi_i\eta_i \right)-\frac{1}{2}.$$
By differentiating the previous equation with respect to  $\xi_1$, we get, for any $k\in\Z^+$,
$$\frac{\de^k H_b}{\de \xi_1^{k}}=\frac{b^k}{2^{k+1}}\exp\left( \frac{b}{2}\sum_{i=1}^n \xi_i\eta_i \right)\ \eta_1^k\,. $$
Now we assume that the condition \eqref{dip1} of Theorem \ref{main} holds true for any  $\I\subset \N^n$ such that $\#\I=N+1$. In particular, if we consider the subset $\I$ of multi-indices \eqref{II}, we have that
$$\sum_{k=1}^{N+1} a_k(\xi) \frac{\de^k H_b}{\de \xi_1^{k}}(\xi,\eta)= \exp\left( \frac{b}{2}\sum_{i=1}^n \xi_i\eta_i \right) \sum_{k=1}^{N+1}\frac{b^k  a_k(\xi) }{2^{k+1}} \eta_1^k \equiv 0,$$
that implies $a_k(\xi)\equiv 0$ for any  $k=1,\mathellipsis,N+1$.
In view of  Theorem  \ref{main}, we  conclude that there are no local \pk immersions of a flat \pk space forms into a non-flat one.\\

\medskip
\noindent
$\bullet$ 
If $c\neq 0$ and $b\neq 0$, by taking into account   \eqref{H} and that the diastasis $D_0^{\s_c^n}$ is given by \eqref{dpn}, we have that
$$ H_b(\xi,\eta)= \frac{1}{2}\left(1+2\sum_{i=1}^n \xi_i\eta_i\right)^\frac{b}{c}-\frac{1}{2}.$$
By differentiating the previous equation with respect to  $\xi_1$, we get, for any $k\in\Z^+$,
$$\frac{\de^k}{\de \xi_1^{k}} H_b= 
2^{k-1} \prod_{j=0}^{k-1}\left(\frac{b}{c}-j\right)\left(1+2\sum_{i=1}^n \xi_i\eta_i\right)^{\frac{b}{c} -k}\eta_1^k\,.$$
Now we assume that the condition \eqref{dip1} of Theorem \ref{main} holds true for any  $\I\subset \N^n$ such that $\#\I=N+1$. In particular, if we consider the subset $\I$ of multi-indices \eqref{II}, we have that
$$\sum_{k=1}^{N+1} a_k(\xi) \frac{\de^k H_b}{\de \xi_1^{k}}(\xi,\eta)=\left(1+2\sum_{i=1}^n \xi_i\eta_i\right)^{\frac{b}{c} -N-1}
\sum_{k=1}^{N+1} a_k(\xi) 2^{k-1} \prod_{j=0}^{k-1}\left(\frac{b}{c}-j\right)\left(1+2\sum_{i=1}^n \xi_i\eta_i\right)^{N+1 -k}\eta_1^k \equiv 0.$$
Since the above equality can be seen as a polynomial in $\eta_1,\mathellipsis,\eta_n$ and the coefficient of its monomial in $\eta_1$ of degree 1 is $\frac{b}{c}a_1(\xi)$, it follows that $a_1(\xi)$ needs to be identically zero.
If $\frac{b}{c}\not\in\Z^+$, then we get, by means of similar considerations, that also $a_k(\xi)\equiv 0$ for any  $k=2,\mathellipsis,N+1$. Otherwise, if $\frac{b}{c}\in\Z^+$, then $H_b$ reads as \eqref{problem}, so, in view of Proposition \ref{existence} and Theorem \ref{main}, there exists, for any $i=0,\mathellipsis,n$,
a strongly full local \pk immersion 
$$f_i:\big( \mathcal{U}_i=\{[Z]\in\DP^n\ | \ |Z_i|^2\neq0\},\omega_c\big) \to \left(\s_b^{\binom{n+\frac{b}{c}}{n}-1},\omega_b\right)\,.$$
Indeed, we need at least $\binom{n+\frac{b}{c}}{n}-1$ monomials to describe the polynomial $H_b$.

Finally, we notice that it follows from the Rigidity Theorem \ref{rigidity} the existence of a para-holomorphic isometry  $F$ of the ambient space such that  $f_i|_{\mathcal{U}_i\cap\mathcal{U}_j }= F\circ f_j|_{\mathcal{U}_i\cap\mathcal{U}_j }$. Hence, we have that any local immersion $f_i$ can be extended to a global immersion.
\end{proof}

\subsection*{Acknowledgements:}

G.~Manno gratefully acknowledges support by  ``Finanziamento alla Ricerca'' \texttt{53\_RBA17MANGIO} and \texttt{53\_RBA21MANGIO},  and PRIN project 2022 ``Real and Complex Manifolds: Topology, Geometry and
holomorphic dynamics'' (code \texttt{2022AP8HZ9}).  The second author is supported by the ``Starting Grant'' under the contact number \texttt{53\_RSG22SALFIL}.  Both authors are members of GNSAGA
of INdAM.

\vspace{1.5cm}
\noindent
\textsc{(G. Manno) Dipartimento di Scienze Matematiche ``G. L. Lagrange'', Politecnico di Torino, Corso Duca degli Abruzzi 24, 10129 Torino.}\\
\textit{Email address:} \texttt{giovanni.manno@polito.it}\\

\noindent
\textsc{(F. Salis) Dipartimento di Scienze Matematiche ``G. L. Lagrange'', Politecnico di Torino, Corso Duca degli Abruzzi 24, 10129 Torino.}\\
\textit{Email address:} \texttt{filippo.salis@polito.it}
\end{document}